\documentclass[reqno]{amsart}
\usepackage{amsthm}
\usepackage{amsmath,amsfonts,amssymb,epsfig,amsbsy, enumerate}
\def\Isom{\operatorname{Isom}}
\newtheorem{thm}{Theorem}[section]
\newtheorem{lemma}[thm]{Lemma}
\newtheorem{prop}[thm]{Proposition}

\newtheorem{cor}[thm]{Corollary}
\theoremstyle{definition}
\newtheorem{defi}[thm]{Definition}

\numberwithin{equation}{section}

\begin{document}
\title[The Margulis region and screw parabolic elements]{The Margulis region and screw parabolic elements of bounded type}
\date{\today}
\author{Viveka Erlandsson}
\address{Department of Mathematics,
The City University of New York,
Graduate Center, New York, NY 10016}
\email{verlandsson@gc.cuny.edu}
\maketitle

\begin{abstract}
Given a discrete subgroup of the isometries of $n$-dimensional hyperbolic space there is always a region kept precisely invariant under the stabilizer of a parabolic fixed point, called the Margulis region. While in dimensions 2 and 3 this region is a horoball, it has in general a more complicated shape due to the existence of screw parabolic elements in higher dimensions. In fact, in a discrete group acting on hyperbolic 4-space containing a screw parabolic element with irrational rotation, the corresponding Margulis region does not contain a horoball~\cite{Sus02}. In this paper we describe the asymptotic behavior of the boundary of the Margulis region when the irrational screw parabolic is of bounded type. As a corollary we show that the region is quasi-isometric to a horoball. Although it is shown in \cite{Kim11} that two screw parabolic isometries with irrational rotation are not conjugate by any quasi-isometry of $\mathbb{H}^4$, this corollary implies that their corresponding Margulis regions (in the bounded type case) are quasi-isometric.
\end{abstract}
\section{Introduction}

Let $G$ be a discrete subgroup of $\Isom(\mathbb{H}^n)$, the group of orientation preserving isometries of hyperbolic $n$-space. We are interested in regions in $\mathbb{H}^n$ which are precisely invariant under the stabilizer $G_\alpha$ of a parabolic fixed point $\alpha$ in $G$.

In dimensions 2 and 3 there is always a precisely invariant horoball for each parabolic fixed point. This follows from Shimizu's lemma \cite{Shi63}: the radii of isometric spheres of the elements in $G$ not fixing $\alpha$ are uniformly bounded by a constant.
Although Shimizu's lemma generalizes to higher dimensions in the special case where the parabolic element $g$ is a pure parabolic (i.e. conjugate to a translation), it does not hold in general. In hyperbolic space of dimension 4 and higher there are parabolic elements with a rotational part, called screw parabolic elements. Ohtake \cite{Oht85} showed that when $G$ contains a screw parabolic element with irrational rotational part there is no uniform bound on the radii of the isometric spheres. In fact, Apanasov~\cite{Apa85} has shown that in discrete groups containing such an element, there may not be a precisely invariant horoball.

In all dimensions there is a region precisely invariant under the stabilizer of a parabolic fixed point $\alpha$, called the \emph{Margulis region}. This region consists of the points in $\mathbb{H}^n$ that are moved at most a distance $\epsilon$ by an element of infinite order in the stabilizer of $\alpha$. When $\epsilon$ is small enough (smaller than the Margulis constant) it follows from Margulis Lemma that this region is precisely invariant under $G_\alpha$. In dimensions 2 and 3 the Margulis region corresponding to a parabolic fixed point is a horoball. In higher dimensions it has in general a more complicated shape. In~\cite{Sus02} Susskind gives an explicit description of the Margulis region in $\mathbb{H}^4=\{(x,y,z,u)\,:\,u>0\}$ for a discrete group containing an irrational screw parabolic element. The shape of the boundary of this region is given by a function $u=b(r)$ where $r~=~\sqrt{x^2+y^2}$ and is related to the continued fraction representation of the irrational angle of rotation. In this paper we use Susskind's results to further describe the shape of the Margulis region when the irrational rotation is of \emph{bounded type}, i.e. when the partial quotients in its continued fraction representation are uniformly bounded. We first describe the asymptotic behavior of the boundary of the region by showing that the function $b(r)$ is comparable to $\sqrt{r}$ (Theorem \ref{thm:comp}). As a consequence we show that there is a quasi-isometry of $\mathbb{H}^4$ which maps the Margulis region to a horoball and we say that the two regions are \emph{quasi-isometric in $\mathbb{H}^4$} (Corollary \ref{cor:main}). It should be noted that two irrational screw parabolic elements are not conjugate to each other by any quasi-isometry of $\mathbb{H}^4$, as shown in \cite{Kim11}. However, Corollary \ref{cor:main} implies that the Margulis regions of two irrational screw parabolic elements, in the bounded type case, are quasi-isometric in $\mathbb{H}^4$.

The author wishes to thank Perry Susskind and Saeed Zakeri for helpful conversations, comments, and corrections which greatly aided the completion of this paper. The author would also like to thank her advisor Ara Basmajian for his help and encouragement throughout this work. The results of this paper are contained in the doctoral dissertation \cite{Erl12}.

\section{Background}

Let $\mathbb{H}^4=\{(x,y,z,u)\in\mathbb{R}^4\,:\,u>0\}$ denote the upper half space model of hyperbolic 4-space, and $\Isom(\mathbb{H}^4)$ its orientation-preserving isometry group. We identify $\Isom(\mathbb{H}^4)$ with the orientation-preserving M\"{o}bius group of the boundary at infinity $\hat{\mathbb{R}}^3$, with the usual Poincar\'{e} extension to $\mathbb{H}^4$.  An isometry is called parabolic if it fixes exactly one point on the boundary. A parabolic isometry $g$ is conjugate (in the M\"{o}bius group) to an element of the form $p\mapsto Ap+a$ where $A\in SO(3)$ and $a\in\mathbb{R}^3-\{0\}$. If $A=I$, then this is a translation and we call $g$ a pure parabolic element. Otherwise we call $g$ a screw parabolic element. We say that $g$ is an irrational screw parabolic element if $A$ has infinite order and rational if $A\neq I$ has finite order. If $g$ is screw parabolic, its axis of rotation is the unique Euclidean line kept invariant under its action on the boundary and $g$ rotates around and translates along this axis. When the angle of rotation is an irrational multiple of $2\pi$, $g$ is an irrational screw parabolic element. (See \cite{Bea83} and \cite{Mas88} for more background on hyperbolic space.)

A (closed) horoball based at $\infty$ is a region of the form $\left\{(x,y,z,u)\in\mathbb{H}^4\,|\,u\geq t\right\}$ for some $t>0$, called the height of the horoball. A horoball based at a finite point $\alpha\in\mathbb{R}^3$ is the isometric image of such a region and hence is bounded by a Euclidean sphere tangent to $\partial\mathbb{H}^4$ at $\alpha$.

Suppose $G$ is a subgroup of $\Isom(\mathbb{H}^4$) containing a parabolic element $g$ with fixed point $\alpha$. Let $G_\alpha$ be the stabilizer of $\alpha$ in $G$. A set $X$ is said to be precisely invariant under $G_\alpha$ in $G$ if $h(X)=X$ for all $h\in G_\alpha$ and $f(X)\bigcap X=\emptyset$ for all $f\in G-G_\alpha$.

Let $\epsilon>0$ be given. The Margulis region corresponding to the parabolic fixed point $\alpha$ is defined to be the set of points
\[\{P\in\mathbb{H}^4\,:\,\rho(P,g(P))\leq\epsilon\,\, \text{for some }\,g\in G_\alpha\, \text{ of infinite order} \}\]
where $\rho(\cdot,\cdot)$ is the hyperbolic metric. Let $\epsilon(n)$ denote the Margulis constant. As a consequence of the Margulis Lemma, if $\epsilon<\epsilon(n)$ this region is precisely invariant under $G_\alpha$ (see \cite{Bow93}).

Since the Margulis region only depends on the infinite order elements in $G_{\alpha}$ and parabolic and loxodromic elements cannot share a fixed point in a discrete group \cite{Bow93}, we can assume that $G_{\alpha}$ only contains parabolic elements. Moreover, a parabolic subgroup containing an irrational screw parabolic element must be of rank one, and hence generated by such an element \cite{Kim11}. If $G_{\alpha}$ is generated by pure parabolic elements then the Margulis region is a horoball. If $G_{\alpha}$ contains a rational screw parabolic then it has a finite index subgroup for which the Margulis region is a horoball (the Margulis region for $G_{\alpha}$ is in general slightly different from a horoball and we will briefly return to this case below). Hence the interesting case is when $G_{\alpha}$ contains an irrational screw parabolic element $g$, and since it is necessarily of rank 1 we will assume that $G_{\alpha}=<g>$. Therefore, the Margulis region corresponding to $G_{\alpha}$ is

\[T_{g}=\{P\in\mathbb{H}^4\,:\, \rho(P,g^k(P))\leq\epsilon \,\,\text{for some}\,\, k\in\mathbb{N}\}.\]

By conjugating if necessary, we can assume that $g$ has fixed point $\infty$ and the $z$-axis as its axis of rotation. Set
\[g(x,y,z,u)=(x\cos\theta-y\sin\theta, x\sin\theta+y\cos\theta, z+\sqrt{2},u)\]
where $\theta\in(0,2\pi)$ is an irrational multiple of $2\pi$.

In \cite{Sus02} Susskind gives an explicit description for the boundary of $T_g$, which we will describe next. Using the hyperbolic distance formula (see \cite{Bea83})
\[\cosh\left(\rho(P,Q)\right)=1+\frac{|P-Q|}{2uu_1}\]
where $P=(x,y,z,u)$ and $Q=(x_1,y_1, z_1, u_1)$ it follows that
\[\cosh\left(\rho(P,g^k(P))\right)=1+\frac{(1-\cos k\theta)r^2+k^2}{u^2}\]
where $r=\sqrt{x^2+y^2}$. Since $\rho(\cdot,\cdot)\leq\epsilon$ if and only if $\cosh\left(\rho(\cdot,\cdot)\right)\leq 1+E$ for some $E>0$, $P\in T_g$ if and only if the above quantity is less than or equal to $1+E$ for some $k$. It follows that $P\in T_g$ if and only if
\[Eu^2\geq (1-\cos k\theta)r^2+k^2\]
for some $k$ and hence the boundary of the Margulis region is given by
\[\partial T_{g} = \{(x,y,z,u)\in \mathbb{H}^4\,:\, u=b(\sqrt{x^2+y^2})\}\]
where
\[b(r)=\inf_{k\in\mathbb{N}}\{u_k(r)\}\]
and
\[u_k(r)=\frac{1}{\sqrt{E}}\sqrt{(1-\cos k\theta)r^2+k^2}\]
for $r\geq0$, $k\in\mathbb{N}$ (see \cite{Sus02} for more details). The positive constant $E$ depends only on the Margulis constant $\epsilon(n)$.

Before continuing describing the function $b(r)$ when the rotational angle is irrational, we will briefly discuss the case when $g$ is a rational screw parabolic element. Clearly the computations above still hold for such an element. Let $l$ be the finite order of the rational rotation of $g$. Note that $u_k(r)\geq k/\sqrt{E}$ for all $k, r\geq 0$ and $u_k(r)\to\infty$ as $r\to\infty$ for all $0<k<n$. Since $u_l(r)=l/\sqrt{E}$ it follows that $b(r)=l/\sqrt{E}$ for all large $r$. Hence $b(r)$ is eventually horizontal and the Margulis region is asymptotic to a horoball.

We now return to the case where $g$ is an irrational screw parabolic.
Some basic properties of the functions $u_k(r)$ follow:

\begin{lemma}[\cite{Sus02} Lemma 4]\label{lemma:Su}
For each $k$, the function $u_k(r)$ is uniformly continuous, differentiable, strictly increasing, convex, and,
\begin{enumerate}
  \item $\{u_k(0)\}$ is an increasing sequence
  \item the collection of graphs $\{u_k(r)\}$ do not accumulate anywhere
  \item the graphs of $u_m(r)$ and $u_k(r)$ intersect at most once
  \item for $m>k$, $u_m(r_0)=u_k(r_0)$ if and only if $\cos(m\theta)>\cos(k\theta)$ and \[r_0^2=\displaystyle\frac{m^2-k^2}{\cos(m\theta)-\cos(n\theta)}\]
  \item if $m>k$ and $u_m(r_0)=u_k(r_0)$ then $u_k(r)<u_m(r)$ for all $r<r_0$ and $u_k(r)>u_m(r)$ for all $r>r_0$
\end{enumerate}
\end{lemma}

The function $b(r)$ is uniformly continuous and strictly increasing. In fact, $b(r)\to\infty$ as $r\to\infty$. Hence, in particular, the Margulis region contains no horoball. Moreover:

\begin{thm}[\cite{Sus02} Theorem 6]
There is a sequence $\{r_m\}$ of real numbers $0=r_0<r_1<\cdots<r_m<\cdots$ where $r_m\to\infty$, and a strictly increasing sequence $\{k_m\}$ of positive integers such that the function $b(r)$
\begin{enumerate}
  \item consists of pieces of $u_k(r)$ for infinitely many $k$, that is,
  \[b(r)=u_{k_m}(r)\,\, \mbox{for}\,\, r\in[r_{m-1},r_m];\]
  \item is locally convex and is differentiable except at the countably many points $r_m$. At these points $u_{k_{m-1}}(r_m)=u_{k_m}(r_m)$;
  \item appears concave in the large, that is, $b'(r):=u_{k_m}'(r+)\to 0$ as $r\to\infty$ and $b'(r)>0$ for all $r$.
\end{enumerate}
\end{thm}

If for some $k$ the curve $u_k(r)$ meets $b(r)$ at an isolated point (which is possible if more than two curves meet at a point) then we will not consider $u_k(r)$ as a constituent part of $b(r)$.

There is a connection between the constituent pieces of $b(r)$ and continued fractions. Let $x$ be the irrational number $\theta/2\pi$ and $[a_0, a_1, a_2, a_3, ...]$ be $x$ expressed as an infinite continued fraction. We say $x$ is of \emph{bounded type} if the partial quotients $a_n$ are uniformly bounded, i.e. there exists a constant $D$ such that $a_n\leq D$ for all $n$. The rational number $x_n=p_n/q_n=[a_0, a_1, a_2, ..., a_n]$, for $n\geq0$, is called the $n^{th}$ convergent of $x$. The sequence of denominators $\{q_n\}$ of the convergents
is a strictly increasing sequence of positive integers, i.e.
\[0<q_n<q_{n+1}\]
for all $n$. Also, by defining $q_{-2}=1$ and $q_{-1}=0$, the denominators satisfy the recursive formula (see \cite{PZ09})
\begin{align}
q_{n+1}=a_{n+1} q_n+q_{n-1}. \label{recursive}
\end{align}
The functions $u_{k_m}(r)$ that appear as a constituent piece of $b(r)$ are related to the continued fraction expansion of $x$ in the following way:

\begin{thm}[\cite{Sus02} Theorem 8]\label{thm:index}
If $b(r)=u_q(r)$ for some $r>0$, then $q$ is the denominator of a convergent of $x$.
\end{thm}

Hence, if $u_{k_{m-1}}(r)$ and $u_{k_m}(r)$ are two consecutive constituent pieces of $b(r)$, then $k_{m-1}=q_l$ and $k_m=q_n$ for some $l<n$.

\section{Asymptotic behavior of $b(r)$}

We are interested in the coarse shape of the Margulis region and hence we will look at the asymptotic behavior of the function $b(r)$. As mentioned above, Susskind \cite{Sus02} showed that the graph of $b(r)$ appears concave in the large. Here we will show, when $\theta$ is of bounded type, that the graph roughly behaves like $\sqrt{r}$. More precisely, we will show that $b(r)/\sqrt{r}$ is bounded for large $r$ (see Theorem \ref{thm:comp}).

By Theorem \ref{thm:index} the indices $k$ of the curves $u_k(r)$ that appear as constituent pieces of $b(r)$ are denominators of the convergents of $x=\theta/2\pi$. As before, let $[a_0, a_1, a_2, a_3, ...]$ be the continued fraction expansion of $x$ and denote the denominator of the $n^{th}$ convergent by $q_n$. For ease of notation we will from here on denote the curve $u_{q_n}(r)$ by $v_n(r)$, that is
\[v_n=\frac{1}{\sqrt{E}}\sqrt{(1-\cos q_n\theta)r^2+q_n^2}.\]
For integers $k>0$ define
\[\|k\theta\|=\min\{|k\theta-2\pi p|\,:\, p\in\mathbb{Z}\}.\]
Since $\theta$ is an irrational multiple of $2\pi$ any orbit under rotation by $\theta$ is dense on the unit circle. In particular, $\|k\theta\|$ comes close to 0 (equivalently, $e^{ik\theta}$ comes close to 1) for infinitely many values of $k$. A \emph{closest return moment} \cite{PZ09} of the orbit $\{e^{ik\theta}\}$ is an integer $q>0$ such that $\|q\theta\|<\|k\theta\|$ for all $0<k<q$. Since $\theta$ is irrational, there are infinitely many closest return moments, in fact:

\begin{thm}[\cite{PZ09}] The denominators $\{q_n\}$ of the convergents of $x=\theta/2\pi$ constitute the closest return moments of any orbit under rotation by $\theta$.
\end{thm}

\noindent This fact is the main ingredient in proving that the constituent pieces of $b(r)$ are indexed by denominators of convergents (Theorem \ref{thm:index}). Fix $r$ and suppose $b(r)=u_q(r)$. Then, for all $k<q$, we have $u_q(r)\leq u_k(r)$ and it follows from the equation defining $u_n$ that $1-\cos q\theta<1-\cos k\theta$. Hence $\|q\theta\|<\|k\theta\|$ for all $k<n$. That is, $q$ is a closest return moment and therefore a denominator of a convergent.

Hence $q_n\theta$ are the angles of interest, and we will use the following facts about their norms:

\begin{lemma}[\cite{PZ09}]\label{lemma:Za}
For every $n\geq 1$,
\begin{enumerate}
  \item $0<\|q_{n+1}\theta\|<\|q_n\theta\|$
  \item $\|q_n\theta\|=a_{n+2}\|q_{n+1}\theta\|+\|q_{n+2}\theta\|$
  \item $\displaystyle\frac{\pi}{q_{n+1}}<\|q_n\theta\|<\frac{2\pi}{q_{n+1}}$
\end{enumerate}
\end{lemma}

Note that for $m>n$, $\|q_{m}\theta\|<\|q_n\theta\|$ implies that $\cos{q_{m}\theta}>\cos{q_{n}\theta}$ so that by Lemma 1, $v_{m}(r)$ and $v_{n}(r)$ must always intersect. We have the following condition for determining whether or not $v_{n}(r)$ is a constituent piece of $b(r)$:

\begin{lemma}\label{lemma:iff}
Fix $n$ and let $r_m$, for each $m$, correspond to the intersection point of $v_{n}(r)$ and $v_{m}(r)$. Then $v_{n}(r)$ does not appear as a constituent piece of $b(r)$ if and only if there exist integers $k, m, k<n<m$ such that $r_k\geq r_m$.
\end{lemma}

\begin{proof}
Suppose there exist integers $k, m, k<n<m$ such that $r_k\geq r_m$. Then by property 5 of Lemma \ref{lemma:Su}, $v_k(r)<v_n(r)$ for $r<r_k$ and $v_{m}(r)<v_{n}(r)$ for $r>r_m$. Since $r_k\geq r_m$ the curve $v(r)=\min\{v_{k}(r), v_{m}(r)\}$ lies strictly below $v_n(r)$ either for all $r$ (if $r_k>r_m$) or for all $r\neq r_k$ (if $r_k=r_m$). In either case $v_{n}(r)$ is not a constituent piece of $b(r)$.

Conversely, suppose $v_{n}(r)$ does not appear as a constituent part of $b(r)$. Then $v_{{k}}(r)$ and $v_{{m}}(r)$ are consecutive pieces of $b(r)$ for some $k<n<m$. Since $k<n<m$ it follows from property 5 of Lemma \ref{lemma:Su} that $v_{n}(r)<v_{{k}}(r)$ for all $r>r_k$ and $v_{n}(r)<v_{{m}}(r)$ for all $r<r_m$. In order for $v_n(r)$ to not be a constituent part of $b(r)$ we must have $r_k\geq r_m$.
\end{proof}

\subsection{Bounded Type}

Throughout this section we will assume that $x=\theta/2\pi$ is of bounded type, i.e. that there is a constant $D$ such that $a_n\leq D$ for all $n$. Under this assumption we will show that the function $b(r)$ is \emph{comparable} to $\sqrt{r}$.

\begin{defi}
Let $F$ and $G$ be functions of $m$. We say that $F$ \emph{is comparable to} $G$ if there exist positive constants $k_1$, $k_2$ and $M$ such that
\[k_1G(m)<F(m)<k_2G(m)\]
for all $m>M$.
\end{defi}

\noindent The constants $k_1$ and $k_2$ will only depend on $D$, the uniform bound on the partial quotients $a_n$. Note that the notion of functions being comparable is a transitive property.

The bounded type assumption implies that the denominators of the convergents of $x$ do not grow too fast:

\begin{lemma}\label{lemma:ratio}
Fix an integer $k>0$. Then $q_{n+k}$ is comparable to $q_n$. In particular, $q_n<q_{n+k}<(D+1)^k q_n$.
\end{lemma}

\begin{proof}
 Fix an integer $k>0$. Clearly, since $\{q_n\}$ is an increasing sequence, $q_{n}<q_{n+k}$. Moreover, using $a_n\leq D$ for all $n$ together with the recursive formula \eqref{recursive}, we have $q_{n+1}\leq Dq_{n}+q_{n-1}<(D+1)q_n$. By induction it follows that $q_{n+k}<(D+1)^kq_n$.
\end{proof}

We will describe the asymptotic behavior of $b(r)$ by first studying the behavior of certain intersection points of the functions $v_{n}(r)$, namely the intersection of $v_{n}(r)$ and $v_{{n+2}}(r)$. Denote the $r-$coordinate of this point by $r_n$, that is:

\[r_n^2=\frac{q_{n+2}^2-q_n^2}{\cos (q_{n+2}\theta)-\cos (q_{n}\theta)}.\]\\

\begin{prop}\label{prop:r}
$r_n$ is comparable to $q_{n}^2$
\end{prop}

The proof of the proposition will follow immediately from the following two lemmas, where we look at the numerator and denominator in the expression for $r_n^2$ separately.

\begin{lemma}\label{lemma:numerator}
$q_{n+2}^2-q_n^2$ is comparable to $q_{n}^2$
\end{lemma}

\begin{proof}
Note that
\[q_{n+2}^2-q_{n}^2=(q_{n+2}+q_{n})(q_{n+2}-q_{n}).\]
Clearly, since $\{q_n\}$ is an increasing sequence,
\begin{align}
q_{n}<q_{n+2}+q_{n}<2q_{n+2}\label{plus}
\end{align}
and hence, using the inequality from Lemma \ref{lemma:ratio}, we have
\[q_{n}<q_{n+2}+q_{n}<2(D+1)^2q_{n}.\]
Also, by the recursive formula \eqref{recursive}, $q_{n+2}-q_{n}=a_{n+2}q_{n+1}$ where $1\leq a_{n+2}\leq D$ and so
\[q_{n+1}\leq q_{n+2}-q_{n}\leq Dq_{n+1}.\]
Using the fact that $q_n<q_{n+1}$ together with Lemma \ref{lemma:ratio} we have
\begin{align}
q_n < q_{n+2}-q_{n} < D(D+1)q_n.\label{minus}
\end{align}
Putting \eqref{plus} and \eqref{minus} together we have
\[q_{n}^2<q_{n+2}^2-q_{n}^2<2D(D+1)^3q_{n}^2\]
as desired.\\
\end{proof}

\begin{lemma}\label{lemma:denominator}
$\cos (q_{n+2}\theta)-\cos (q_{n}\theta)$ is comparable to $1/q_{n}^2$
\end{lemma}

\begin{proof}
Suppose $0<a<b<\pi/2$. By the mean value theorem,
\[\frac{\cos a - \cos b}{a-b}=\sin c\]
for some $a<c<b$.
Note that
\[\frac{2}{\pi}<\frac{\sin c}{c}<1\]
for $0<c<\pi/2$.
Choose $N$ such that $||q_n\theta||<\pi/2$ for all $n>N$. Letting $n>N$, $a=||q_{n+2}\theta||$, $b=||q_{n}\theta||$, and using Lemma \eqref{lemma:Za} we have
\begin{align}
\frac{2}{\pi}||q_{n+2}\theta||<\frac{\cos(q_{n+2}\theta)-\cos(q_n\theta)}{a_{n+2}||q_{n+1}\theta||}<||q_n\theta||.\label{cos}
\end{align}
Therefore
\[\cos(q_{n+2}\theta)-\cos(q_n\theta)<a_{n+2}||q_{n}\theta||||q_{n+1}\theta||\]
and by property 3 of Lemma \ref{lemma:Za} together with the bounded type assumption we get
\[\cos(q_{n+2}\theta)-\cos(q_n\theta)<\frac{4D\pi^2}{q_{n+1}q_{n+2}}<\frac{4D\pi^2}{q_n^2}.\]

Also by \eqref{cos},
\[\cos(q_{n+2}\theta)-\cos(q_n\theta)>\frac{2}{\pi}a_{n+2}||q_{n+1}\theta||||q_{n+2}\theta||\]
and hence by property 3 of Lemma \ref{lemma:Za} and the inequality from Lemma \ref{lemma:ratio},

\[\cos(q_{n+2}\theta)-\cos(q_n\theta)>\frac{2\pi}{q_{n+2}q_{n+3}}>\frac{2\pi}{(D+1)^5q_n^2}.\]

\end{proof}

\noindent We can now describe the asymptotic behavior of $b(r)$:

\begin{thm}\label{thm:comp}
The function $b(r)$ is comparable to $\sqrt{r}$.
\end{thm}

\begin{proof}
By Proposition \ref{prop:r} choose $N$ such that
\[k_1q_{n}^2<r_n<k_2q_{n}^2\]
for all $n>N$, where $r_n$ as before corresponds to the intersection of $v_{n}(r)$ and $v_{{n+2}}(r)$ and $k_1, k_2$ are positive constants.
Choose $R$ large enough so that $b(R)=v_{m}(R)$ for some $m>N$ and fix $r>R$. Suppose $b(r)=v_{n}(r)$. (If $r$ corresponds to an intersection point of two or more curves, choose $v_{n}$ to be the curve such that $b(r+)=v_{n}(r+)$).

Since $v_{n}$ is a constituent piece of $b(r)$ it follows by Lemma \ref{lemma:iff} that $r_{n-2}\leq r\leq r_n$. Hence, by Proposition \ref{prop:r},

\[k_1q_{n-2}^2<r<k_2q_n^2.\]

\noindent In fact, by Lemma \ref{lemma:ratio}, $q_{n-2}>q_{n}/(D+1)^2$ and hence

\begin{align}
\frac{k_1}{(D+1)^4}q_n^2<r<k_2q_n^2.\label{ineq}
\end{align}

In particular $q_n^2>r/k_2$ and we immediately get the lower bound

\[b(r)=\frac{1}{E}\sqrt{(1-\cos q_n\theta)r^2+q_n^2} > \frac{1}{E}q_n > \frac{1}{E\sqrt{k_2}}\sqrt{r}.\]

For the upper bound, note that for $||q_n\theta||<\pi/2$
\[1-\cos q_n\theta < \frac{||q_n\theta||^2}{2}\]
and hence, using Lemma \ref{lemma:Za} and the fact that $\{q_n\}$ is an increasing sequence,

\[1-\cos q_n\theta < \frac{2\pi^2}{q_{n+1}^2} < \frac{2\pi^2}{q_n^2}.\]

It follows from the above inequality together with the upper bound for $r$ in \eqref{ineq} that

\[b(r)=\frac{1}{E}\sqrt{(1-\cos q_n\theta)r^2+q_n^2} < \frac{1}{E} \left(\sqrt{2k_2^2\pi^2+1}\right)q_n\]

Also, from the lower bound in \eqref{ineq} we have $q_n<((D+1)^2/\sqrt{k_1})\sqrt{r}$ and hence
\[b(r) < \frac{1}{E} \left(\sqrt{2k_2^2\pi^2+1}\right)\left(\frac{(D+1)^2}{\sqrt{k_1}}\right)\sqrt{r}.\]

\end{proof}

As has been the case throughout this section, the constants bounding $b(r)/\sqrt{r}$ (for large $r$) only depend on $D$. However, the constant $R$ in the proof above (unlike the constant $N$) could in general depend on the particular angle $\theta$.


\section{Quasi-Isometries and the Margulis Region}

Let $\mathcal{B}=\{(x,y,z,u)\in\mathbb{H}^4\,:\,u\geq 1\}$ denote the (closed) horoball based at $\infty$ of height 1 in $\mathbb{H}^4$. As a corollary to Theorem \ref{thm:comp} we will show that the Margulis region $T_g$, where $g$ is an irrational screw parabolic element of bounded type, is \emph{quasi-isometric to $\mathcal{B}$ in $\mathbb{H}^4$}.

\begin{defi}
A map $f: X\to Y$ between two metric spaces $(X, d_1)$ and $(Y,d_2)$ is a \emph{quasi-isometry} if there exist constants $\lambda\geq1$ and $C,M\geq0$ such that
\begin{enumerate}
  \item for all $x,y\in X$, $\frac{1}{\lambda}d_1(x,y)-C<d_2(f(x),f(y))<\lambda d_1(x,y)+C$, and
  \item for all $y\in Y$, $d_2(y,f(x))<M$ for some $x\in X.$
\end{enumerate}
We say that two submanifolds $V$ and $W$ of $\mathbb{H}^4$ are \emph{quasi-isometric in $\mathbb{H}^4$} if there exists a quasi-isometry $f:\mathbb{H}^4\to\mathbb{H}^4$ such that $f(V)=W$.
\end{defi}

Theorem \ref{thm:comp} shows, in the bounded type case, that $b(r)/\sqrt{r}$ is bounded for sufficiently large $r$. Clearly it is not bounded for all $r\geq0$ since there is no upper bound for $r$-values in a neighborhood of 0. However, if we replace $\sqrt{r}$ by a function that is asymptotic to $\sqrt{r}$ as $r\to\infty$ and for which this ratio is bounded for small $r$, we can extend the bound to all non-negative $r$. Let $a(r)$ be such a function. Then $\sqrt{r}/a(r)$ is bounded for large $r$. Hence, since
\[\frac{b(r)}{a(r)}=\frac{b(r)}{\sqrt{r}}\frac{\sqrt{r}}{a(r)}\]
there exists a constant $R$ such that $b(r)/a(r)$ is bounded by positive constants for all $r>R$.
By the assumption that $b(r)/a(r)$ is bounded for small $r$ we have
\[A<\frac{b(r)}{a(r)}<B\]
for all $r\geq0$ for some constants $A,B>0$.

For any function $a(r)$ with the properties described, define
\[S_a=\{(x,y,z,u)\in\mathbb{H}^4\,:\, u\geq a (\sqrt{x^2+y^2})\}.\]
We will first show that $T_g$ and $S_a$ are quasi-isometric in $\mathbb{H}^4$.

\begin{lemma}\label{lemma:qi}
Let $g$ be an irrational screw parabolic element of bounded type. Then the Margulis region $T_{g}$ and the region $S_a$ are quasi-isometric in $\mathbb{H}^4$.
\end{lemma}

\begin{proof}
Let $f:\mathbb{H}^4\to\mathbb{H}^4$ such that
\[f(x,y,z,u)=\left(x,y,z,\frac{b(r)}{a(r)}u\right)\]
where $r=\sqrt{x^2+y^2}$. Clearly $f$ is surjective and $f(S_a)=T_{g}$.

As noted above, there exists positive constants $A$ and $B$ such that  $A<b(r)/a(r)<B$ for all $r\geq0$. Let $C=\max\{|\ln{A}|, |\ln{B}|\}$. Then (see \cite{Bea83}), for any point $P=(x,y,z,u)$ in $\mathbb{H}^4$,
\begin{align}
\rho(P,f(P))=\left|\ln{\frac{b(r)}{a(r)}}\right|<C.\label{c}
\end{align}
Let $P$ and $Q$ be any two points in $\mathbb{H}^4$. By the triangle inequality
\[\rho(f(P),f(Q))\leq\rho(f(P),P)+\rho(P,Q)+\rho(Q,f(Q))\]
and hence by \eqref{c}
\[\rho(f(P),f(Q))<\rho(P,Q)+2C.\]
Similarly,
\[\rho(P,Q)\leq\rho(P,f(P))+\rho(f(P),f(Q))+\rho(f(Q),Q)\]
and hence
\[\rho(P,Q)<\rho(f(P), f(Q))+2C.\]
Therefore
\[\rho(P,Q)-2C<\rho(f(P),f(Q))<\rho(P,Q)+2C\]
and $f$ is the desired quasi-isometry.
\end{proof}

We will next find a quasi-isometry of $\mathbb{H}^4$ that maps the horoball $\mathcal{B}$ to a region of the form $S_a$. Let $\sigma>0$ and define $h_{n,\sigma}: \mathbb{H}^n\to \mathbb{H}^n$ such that $h_{n,\sigma}: P\mapsto |P|^{\sigma-1}P$. It is known (see~\cite{Bas00}, \cite{Coo95}) that $h_{n,\sigma}$ is a bilipschitz map onto $\mathbb{H}^n$ with constant $\sigma$ or $1/\sigma$ depending on if $\sigma>1$ or $\sigma<1$, respectively. Consider the following modification of the map in dimension 4:
\[h: (x,y,z,u)\mapsto \lambda(x,y,z,u)\quad \mbox{where}\quad \lambda=\sqrt{x^2+y^2+u^2}.\]
By identifying $\mathbb{H}^3$ with $\{(x,y,z,u)\,|\,z=0\}$ in $\mathbb{H}^4$, note that $h|_{\mathbb{H}^3}=h_{3,2}$ and hence $h$ is bilipschitz with constant 2 if restricted to this slice. We claim that $h$ is in fact bilipschitz in $\mathbb{H}^4$.

\begin{prop}\label{prop:map}
$h: (x,y,z,u)\mapsto \lambda(x,y,z,u)$ where $\lambda=\sqrt{x^2+y^2+u^2}$ is a bilipschitz map of $\mathbb{H}^4$.
\end{prop}

\begin{proof}
Let $P, Q\in\mathbb{H}^4$. By composing with an isometry assume that $P=(0,0,0,\widetilde{u})$ and $Q=(x,y,z,u)$ where $\widetilde{u}\leq u$. Define $\widehat{Q}=(x,y,0,u)$. Note that
\begin{align*}
 h(P)&=(0,0,0,\widetilde{u}^2),\\
 h(Q)&=(\lambda x, \lambda y, \lambda z, \lambda u),\\
 h(\widehat{Q})&=(\lambda x, \lambda y, 0, \lambda u)
\end{align*}
where $\lambda=\sqrt{x^2+y^2+u^2}$.

Since $P, \widehat{Q}\in \mathbb{H}^3$ and $h|_{\mathbb{H}^3}$ is bilipschitz with constant 2,
\begin{align}
\frac{1}{2}\rho(P,\widehat{Q})<\rho(h(P),h(\widehat{Q}))<2\rho(P,\widehat{Q}).\label{bilip}
\end{align}

Using the distance formula (see \cite{Bea83})
\[\rho(R_1,R_2)=2\sinh^{-1}\left(\frac{|R_1-R_2|}{2\sqrt{u_1u_2}}\right)\]
(where $u_1$ and $u_2$ are the $u$-coordinates of the points $R_1$ and $R_2$, respectively), it easily follows that
\begin{align}
\rho(Q,\widehat{Q})=\rho(h(Q),h(\widehat{Q})).\label{q}
\end{align}

Also, using the same formula along with the fact that $\sinh^{-1}(\cdot)$ is a strictly increasing function, simple calculations show that $\rho(P,\widehat{Q})\leq\rho(P,Q)$ and, since $u\geq \widetilde{u}$, that $\rho(\widehat{Q},Q)\leq\rho(P,Q)$. Hence
\begin{align}
\rho(P,Q)\leq\rho(P,\widehat{Q})+\rho(\widehat{Q},Q)\leq2\rho(P,Q).\label{p}
\end{align}

Similarly, using the $\sinh^{-1}(\cdot)$ formula for $\rho(\cdot,\cdot)$ we easily observe that $\rho(h(P),h(\widehat{Q}))\leq\rho(h(P),h(Q))$, and since $\lambda u\geq u^2\geq\widetilde{u}^2$ we also observe that $\rho(h(\widehat{Q}),h(Q))\leq\rho(h(P),h(Q))$. It follows that
\begin{align}
\rho(h(P),h(Q))\leq\rho(h(P),h(\widehat{Q}))+\rho(h(\widehat{Q}),h(Q))\leq2\rho(h(P),h(Q)).\label{hp}
\end{align}

We now have:
\begin{align*}
  \rho(h(P), h(Q))&\leq\rho(h(P), h(\widehat{Q}))+\rho(h(\widehat{Q},h(Q)) &\quad \mbox{(by \eqref{hp})}\\
  &\leq2\rho(P,\widehat{Q})+\rho(\widehat{Q},Q) &\mbox{(by \eqref{bilip} and \eqref{q})}\\
  &\leq2\left(\rho(P,\widehat{Q})+\rho(\widehat{Q},Q)\right)\\
  &\leq4\rho(P,Q)&\mbox{(by \eqref{p})}
\end{align*}
and
\begin{align*}
  \rho(h(P), h(Q))&\geq\frac{1}{2}\left(\rho(h(P), h(\widehat{Q}))+\rho(h(\widehat{Q}), h(Q))\right)&\quad \mbox{(by \eqref{hp})}\\
  &\geq\frac{1}{2}\left(\frac{1}{2}\rho(P,\widehat{Q})+\rho(\widehat{Q},Q)\right)&\mbox{(by \eqref{bilip} and \eqref{q})}\\
  &\geq\frac{1}{4}\left(\rho(P,\widehat{Q})+\rho(\widehat{Q},Q)\right)\\
  &\geq\frac{1}{4}\rho(P,Q) &\mbox{(by \eqref{p})}
\end{align*}
Hence $h$ is bilipschitz.
\end{proof}

A calculation shows that for $r'=r\sqrt{r^2+1}$ an inverse function is given by
\[s_r=\sqrt{\frac{\sqrt{4r^2+1}-1}{2}}.\]
Define $a(r)=\sqrt{s_r^2+1}$, that is
\[a(r)=\sqrt{\frac{\sqrt{4r^2+1}+1}{2}}.\]
It is easily verified that $a(r)$ is asymptotic to $\sqrt{r}$. Also, since $a(r)$ is continuous and $a(r)\geq1$ for $r\geq0$ it is clear that $b(r)/a(r)$ is bounded in any bounded neighborhood of 0. Hence, by Lemma \ref{lemma:qi}, $S_a$ is quasi-isometric in $\mathbb{H}^4$ to the Margulis region $T_g$. We will show that the bilipschitz map $h$ in Proposition \ref{prop:map} maps the horoball $\mathcal{B}$ onto $S_a$. Hence, composing the quasi-isometry $f$ from Lemma \ref{lemma:qi} with the bilipschitz map $h$ results in a quasi-isometry of $\mathbb{H}^4$ mapping the horoball $\mathcal{B}$ to the Margulis region.

In order to simplify notations we employ the standard cylindrical coordinates \\$<r,\psi, z>$ instead of the Cartesian coordinates $(x,y,z)$. The image of the horosphere $\mathcal{U}$ of height 1 (the boundary of $\mathcal{B}$) under $h$ is the set
\begin{align*}
  &\left\{\lambda(x,y,z,1)\,:\, \lambda=\sqrt{x^2+y^2+1}\right\}\\
  &=\left\{<\lambda r,\psi,\lambda z,\lambda>\,:\, \lambda = \sqrt{r^2+1}\right\}\\
  &=\left\{<r',\psi,z',u'>\,:\, u'=a(r') \right\}
\end{align*}
Hence
\[h(\mathcal{B})=\{(x,y,z,u)\,:\,u\geq a(\sqrt{x^2+y^2}) \}\]
that is, $h(\mathcal{B})=S_a.$
We have proved:

\begin{cor}\label{cor:main}
Let $g$ be an irrational screw parabolic element of bounded type. Then the Margulis region $T_g$ is quasi-isometric to a (closed) horoball in $\mathbb{H}^4$.
\end{cor}

In \cite{Kim11} it is shown that two irrational screw parabolic elements are not conjugate to each other by any quasi-isometry of $\mathbb{H}^4$. However, Corollary~\ref{cor:main} implies, in the bounded type case, that the corresponding Margulis regions of two such elements are quasi-isometric.

\bibliographystyle{amsalpha}
\bibliography{ref}
\end{document}